\documentclass[reqno, intlimits, sumlimits]{amsart} 

\usepackage{amsmath, amstext, amssymb, amsthm, bbm, lmodern, enumerate, esint, lipsum,color,todonotes}

\usepackage[left=3cm,right=3cm, top=4cm, bottom=3cm]{geometry}

\usepackage[hidelinks]{hyperref}
\usepackage[symbols,nogroupskip,sort=def]{glossaries-extra}

\allowdisplaybreaks

	\renewcommand{\div}{\textup{div}\,}
    \renewcommand{\phi}{\varphi}
    \renewcommand{\epsilon}{\varepsilon}
    \renewcommand{\P}{\mathsf{P}}

	\newcommand	{\eins}	{\mathbbm{1}}   
	\newcommand	{\norm}[1]	{\left\lVert#1\right\rVert}
	\newcommand	{\abs}[1]		{\left\lvert#1\right\rvert}

\newcommand{\Cpl}{ \mathsf{Cpl}}
\newcommand{\cpl}{ \mathsf{cpl}}
\newcommand{\Q}{ \mathsf{Q}}

\newcommand{\q}{ \mathsf{q}}
\newcommand{\C}{\mathsf{C}}

\newcommand{\Leb}{\mathsf{Leb}}
\newcommand{\proj}{\mathsf{proj}}

\newcommand{\cost}{\mathsf c}
\newcommand{\met}{\mathsf d}
\newcommand{\U}{\mathsf U}
\newcommand{\test}{\mathcal{D}}
\newcommand\inner[2]{\langle #1, #2 \rangle}

	\DeclareMathOperator	{\IE}			{\mathbb{E}} 
	\DeclareMathOperator	{\IN}			{\mathbb{N}}
	\DeclareMathOperator	{\IP}			{\mathbb{P}}
	
	\DeclareMathOperator	{\IR}			{\mathbb{R}}

	\DeclareMathOperator	{\supp}			{supp}

        \DeclareMathOperator
    {\PP}			{\mathsf P}
     \DeclareMathOperator
    {\wass}			{\mathcal P(\Omega\times \IR^d)}

	\theoremstyle{plain}
\newtheorem{thm}			{Theorem}[section]
\newtheorem{lem}	[thm]	{Lemma}
\newtheorem{cor}	[thm]	{Corollary}
\newtheorem{prop}	[thm]	{Proposition}

\theoremstyle{definition}
\newtheorem{defi}	[thm]	{Definition}
\newtheorem{ex}	    [thm]	{Example}
\newtheorem{rem}	[thm]	{Remark}

\numberwithin{equation}{section}

\newcommand {\MH}[1]{\todo[inline,size=\footnotesize,color=green]{\textbf{MH:} #1}}

\newcommand {\BM}[1]{\todo[inline,size=\footnotesize,color=cyan]{\textbf{BM:} #1}}

\begin{document}
\title{A Benamou-Brenier formula for transport distances between stationary random measures}

\author[M. Huesmann]{Martin Huesmann}
\address{M.H.: Universit\"at M\"unster, Germany}
\email{martin.huesmann@uni-muenster.de}
\author[B. Müller]{Bastian Müller}
\address{B.M.: Universit\"at M\"unster, Germany   }
\email{bastian.mueller@uni-muenster.de}

\thanks{MH and BM are funded by the Deutsche Forschungsgemeinschaft (DFG, German Research Foundation) under Germany's Excellence Strategy EXC 2044 -390685587, Mathematics M\"unster: Dynamics--Geometry--Structure and   by the DFG through the SPP 2265 {\it Random Geometric Systems. }}

\begin{abstract}
%Let $\xi$ and $\eta$ be two equivariant random measures of equal finite intensity. We prove a Benamou-Brenier formula for a cost function $c(\xi,\eta)$ introduced by Sturm and Huesmann. To that end we introduce a metric and a notion of gradient on probability spaces $\Omega$ admitting a measurable flow, which allows us to formulate a continuity equation on the space $\Omega\times \IR^d$.
We derive a Benamou-Brenier type dynamical formulation for the Wasserstein metric $\mathsf W_p$ between stationary random measures recently introduced in \cite{erbar2023optimal}. A key step is a reformulation of the metric $\mathsf W_p$ using Palm probabilities.
\end{abstract}

\date{\today}
\maketitle

\maketitle

\section{ Introduction}

Let $\mu_0,\mu_1\in\mathcal P_p(\mathbb R^d)$ be two probability measures on $\mathbb R^d$ with finite $p$-th moment, $p>1$. The Kantorovich-Wasserstein distance is 
\[
W^p_p(\mu_0,\mu_1)=\inf_{q\in \mathsf {Cpl}(\mu_0,\mu_1)}\int \norm{x-y}^p q(dx,dy),
\]
where $\mathsf {Cpl}(\mu_0,\mu_1)=\{q\in\mathcal P(\mathbb R^d\times\mathbb R^d): q(A\times \mathbb R^d)=\mu_0(A), q(\mathbb R^d\times A)=\mu_1(A), A\in \mathcal B(\mathbb R^d)\} $ is the set of all couplings between $\mu_0$ and $\mu_1$. A key step in the fascinating development the theory of optimal transport has undertaken in the last 30 years is the dynamical formulation of $W_p$ obtained by Benamou-Brenier in \cite{Benamou2000ACF}:
\begin{align}\label{eq:BB}
     W^p_p(\bar\mu_0,\bar\mu_1)=\inf \left\{\int_0^1\norm{v_t}_{L^p(\mu_t)}^pdt\mid \mu_0=\bar\mu_0,\mu_1=\bar\mu_1,\partial_t\mu_t+\div(v_t\mu_t)=0 \right\}.
\end{align}
For $p=2$ the right hand side, formally, has the structure of a Riemannian metric on $\mathcal P_2(\mathbb R^d)$ inducing a formal calculus that allows to rephrase certain PDEs as gradient flows and leads to new interpretation and proofs of and relation between several functional and geometric inequalities, see e.g.\ \cite{Ot01, OtVi00} \cite[Chapter 4]{figalli2021invitation}, \cite[Chapter 8]{Santa} or \cite[Chapter 23]{Villani}.

The goal of this article is to derive a Benamou-Brenier type dynamical formulation of the metric $\mathsf W_p$ between stationary random measures recently introduced in \cite{erbar2023optimal}. Before stating our main result we will recall the definition of $\mathsf W_p$,  $p>1$. Denote by $\mathcal M(\mathbb R^d)$ the set of all Radon measures on $\mathbb R^d$. Let $\theta=(\theta_x)_{x\in\mathbb R^d}$ be the natural action of $\mathbb R^d$ on $\mathcal M(\mathbb R^d)$ given by $\theta_x\xi(A)=\xi(A-x).$ Any measure $\mathsf P\in\mathcal P(\mathcal M(\mathbb R^d))$ invariant under the action of $\theta$ is called stationary. The intensity of a stationary $\mathsf P$ is given by $\int \xi([0,1]^d)\mathsf P(d\xi)$. We want to define a metric between two stationary random measures of equal finite intensity, which we will assume to be one. To this end, it is useful to take a probabilistic view and consider equivariant random measures instead of stationary probability measures. A random measure is a random variable $\xi:\Omega\to\mathcal M(\IR^d)$ defined on some probability space $(\Omega,\mathcal F, \Q)$. We will assume that the probability space admits a measurable flow $\theta=(\theta_x)_{x\in\IR^d}$\footnote{For notational simplicity we do not highlight the probability space on which the flow is defined. It will always be clear from the context.} such that 
\[
\theta_x\Omega\to\Omega, \theta_0=id, \theta_x\circ\theta_y=\theta_{x+y}, \text{ for all } x,y\in \IR^d,
\]
and $\Q$ is invariant under the flow, i.e.\ $\Q$ is stationary. Moreover, we will assume that on this probability space there are two random measure $\xi,\eta$ both equivariant under the flow, i.e.\ $\xi(\theta_x\omega)(A-x)=\xi(\omega)(A)$, and such that the distribution of $\xi$ is $\mathsf P$ and that of $\eta$ is $\mathsf R$.
We will call such a probability space admissible (see \eqref{eq:topology_on_omega} for the precise definition).  A random measure $\q:\Omega\to\mathcal M(\IR^d \times \IR^d)$ is called an equivariant coupling of $\xi$ and $\eta$, if it is almost surely a coupling of $\xi$ and $\eta$ and if it is invariant under the diagonal action of $\theta$, i.e.\ $\q(\theta_x\omega)(A-x,B-x)=\q(\omega)(A,B)$ for all Borel sets $A,B$.  We denote the set of all equivariant couplings of $\xi$ and $\eta$ by $\mathsf{cpl}_e(\xi,\eta).$ For $n\in\IN$ write $\Lambda_n=[-n/2,n/2]^d.$ Then, we can define the metric $\mathsf W_p$ by
\begin{equation}\label{eq:Wp}
\mathsf W^p_p(\mathsf P,\mathsf R)=\inf_{(\Omega,\mathcal F,\Q), (\xi,\eta)} \inf_{\q\in\mathsf{cpl}_e(\xi,\eta)}\IE_\Q\left[ \int_{\Lambda_1\times \IR^d} \norm{x-y}^p \q(dx,dy)\right] ,
\end{equation}
where the first infimum runs over all admissible probability spaces and all pairs of equivariant random measures $(\xi,\eta)$ such that $\xi\sim \mathsf P $ and $\eta\sim\mathsf R.$  In \cite{erbar2023optimal} it was shown that $\mathsf W_p$  is indeed a geodesic metric. Note that the definition of $\mathsf W_p$ is a two layer optimisation problem.

To derive a dynamical formulation of $\mathsf W_p$ similar to \eqref{eq:BB} we need to find a good replacement of the continuity equation dealing on the one hand with the infinite mass of $\xi$ and $\eta$ and on the other hand with the equivariance. The intuitive idea is as follows: Let $\q\in\mathsf{cpl}_e(\xi,\eta)$. By equivariance,
\begin{align}\label{eq:cost small r}
\IE_\Q\left[ \int_{\Lambda_1\times \IR^d} \norm{x-y}^p \q(dx,dy)\right]  = \frac1{r^d}\IE_\Q\left[ \int_{\Lambda_r\times \IR^d} \norm{x-y}^p \q(dx,dy)\right] 
\end{align}
for any $r>0$. Considering $r\to 0$ corresponds to conditioning $\xi$ to have mass in $0$. Mathematically, this means we consider Palm measures. Secondly, if $\q=(id,T)_\#\xi$, the map $T$ inherits equivariance properties of $\q$, i.e.\ necessarily we have $T(\theta_x\omega)(y)=T(\omega)(y+x)+x.$ In particular, for describing $T$ it is sufficient to specify $T(\omega)(0)$ for all $\omega$. This is exactly the relevant random variable when we consider the $r\to 0$ limit in \eqref{eq:cost small r}, i.e.\ condition on $\xi$ having mass in $0$. In \cite{LaTh09}, Last and Thorisson show that this intuition is correct. It is in fact equivalent to consider equivariant couplings between $\xi$ and $\eta$ or so-called shift-couplings of Palm versions of $\xi$ and $\eta$ under $\Q$ (see Section \ref{sec:metric_formulation} for a precise statement). The main advantage for us is that the coupling between the Palm versions does not require infinite mass, only random variables $T(\omega)(0)$ taking values in $\IR^d$ such that equivariant curves connecting $\xi$ and $\eta$ can be represented by curves of random variables in $\IR^d$ if we are willing to work with Palm measures. Therefore, the continuity equation in the context of stationary random measures can be cast in the form
\[
\partial_t\mathbb P_t + \nabla \cdot (V_t\mathbb P_t)=0
\]
for a curve of probability measures on $\Omega\times \IR^d$, a vector field $V:[0,1]\times \Omega \times \IR^d\to \IR^d\times \IR^d$ and a particular gradient taking the equivariance constraint into account (see Definition \ref{def:cont_equat}). The two components of the measures $\mathbb P_t$ encode on the one hand the evolution of the random measures and on the other hand the curve $(T_t(\omega)(0))_{t\in [0,1]}$ connecting the identity with $T(\omega)(0)$. 

Our main result (see Corollary \ref{cor:BenBrenier_lawsPP} for the precise statement) is as follows:

\begin{thm}\label{thm:main}
 For two stationary random measures $\P_i$, $i=0,1$, with unit intensity, we have for $p>1$
 \begin{align*}
        \mathsf W_p^p(\PP_0,\PP_1)=\inf \int_0^1 \norm{V_t}_{L^p(\IP_t)}^pdt,
    \end{align*}
    where the infimum runs over all admissible probability spaces $(\Omega,\mathcal F,\Q)$ and solutions to the continuity equation $((\IP_t)_{t\in [0,1]},(V_t)_{t\in [0,1]})$
%    curves of equivariant random measures $(\xi_t)_{0\leq t\leq 1}$, defined on $\Omega$, with the following properties 
   such that
    \begin{enumerate}
    \item 
      $(\proj_{\Omega})_{\#}\IP_t=\Q_{\xi_t}$, $0\leq t\leq 1$, where $\Q_{\xi_t}$ is the Palm measure of an equivariant random measures $\xi_t$ under $\Q$, and
        %  $((\IP_t)_{t\in [0,1]},(V_t)_{t\in [0,1]})$ satisfies \eqref{eq:CE},  % such that $(\IP_t)_{t\in [0,1]}$ is a curve of probability measures  and  where $(V_t)_{t\in [0,1]}$ satisfies  the assumptions \eqref{eq:usual_assumptions}.
\item the curve  $((\xi_t)_{\#}\Q)_{t\in [0,1]}$ is weakly continuous and connects
        $\xi_0\underset{\Q}{\sim}\PP_0$ and $\xi_1\underset{\Q}{\sim}\PP_1$.    
    \end{enumerate}
\end{thm}

The proof of this result goes in three steps. First we combine the results of \cite{LaTh09} and \cite{erbar2023optimal} to derive another representation of $\mathsf W_p$ using couplings of Palm measure which is again a two layer optimisation problem. We then first study the inner layer on a fixed admissible probability space. There, we will establish a Benamou-Brenier type representation adapting the approach of  \cite[Chapter 8]{AGS08} to our setup. In particular, we will also establish  a superposition principle  (see Corollary \ref{cor:superpositionpr}) for solutions to the continuity equation and a representation of solutions of the continuity equation by curves of random measures (Proposition \ref{prop:repres_xi_t}). Finally, we derive the main result by taking the infimum over all admissible probability spaces.

\subsection{Connection to the literature.}
Transportation of stationary random measures resp.\ existence and construction of balancing transport kernels, especially matchings and allocations, has been an active field of research in recent years. Results include the close  connection to (shift-) couplings of the corresponding Palm measures (cf.\ \cite{HoPe05,LaTh09}), applications to the Skorokhod embedding problem (cf.\ \cite{Last_Mörters}) and (non)-existence of factor allocations (cf.\ \cite{LaTh23, KhMe23}). Tools from optimal transport have been used in \cite{HS13, H16, HuMu23} to construct allocations between random measures. In the recent article \cite{erbar2023optimal}, the authors construct the metric $\mathsf W_p$ and show that the independent particle movement can be seen as a gradient flow of the specific relative entropy w.r.t.\ $\mathsf W_p$. Moerover, they establish an analogue of the classical HWI inequality.

In the context of non-stationary random measures \cite{GOZLAN2021109141}  prove classical functional inequalities in the setting of random point measures. The article \cite{ErHu15} initiated the investigation of synthetic curvature properties of the configuration space w.r.t.\ the non-normalised Wasserstein distance. This was vastly generalised in \cite{DSSu21,DSSu22} see also \cite{Su22}. In the recent \cite{DSHeSu23}, the authors introduce another metric on the configuration space using a Benamou-Brenier type formulation of a metric based on the non-local differential structure of the add-one cost operator. They show synthetic curvature properties and functional inequalities of the configuration space w.r.t.\ this metric.

The article \cite{DSHeSu23} is a good example of the flexibility of the Benamou-Brenier approach which was previously used for example in discrete spaces \cite{Maa11, Mie11}, for jump processes \cite{Er14, PeRoSaTs22}  or in the context of martingale optimal transport \cite{HuTr19}.

\subsection{Structure of the article.}
Section \ref{sec:setup} is devoted to fixing the setup of this article. In Section \ref{sec:metric_formulation} we  introduce a metric $\met$ on the space $\Omega\times \IR^d$, which allows us to rewrite the transport cost for two random measures $\xi$ and $\eta$ on a fixed admissable probability space as the double infimum $\inf_{\IP_0,\IP_1}\inf_{\U\in \Cpl(\IP_0,\IP_1)}\IE_{\U}\left[\met^2\right]$, where the first infimum ranges over all  measures satisfying $(\proj_{\Omega})_{\#}\IP_0=\Q_{\xi}$ and  $(\proj_{\Omega})_{\#}\IP_1=\Q_{\eta}$. Noting that the inner infimum is equal to the squared $\met$-Wasserstein distance $\mathbb W^2(\IP_0\IP_1)$, in Section \ref{sec:cont_equation}, we follow closely the classical theory.
At the beginning of Section \ref{sec:cont_equation} we introduce a notion of gradient on $\Omega$, which is defined using  the measurable flow on $\Omega$. This allows us  to formulate the continuity equation on $\Omega\times\IR^d$. In the rest of this section we adapt the approach of \cite[Chapter 8]{AGS08} to analyse the continuity equation on $\Omega\times \IR^d$ and to prove the Benamou-Brenier formulation for $\mathsf W_p$, i.e.\ Theorem \ref{thm:main}.

\section{Setup}\label{sec:setup}

Let $\Omega$ be a metrizable topological space,  $\mathcal F$  the Borel $\sigma$-algebra and $\Q$  a probability measure on $\mathcal F$.
We assume that  $(\Omega,\mathcal{F},\Q)$ is   equipped with a measurable flow $\theta_z:\Omega\to \Omega$, $z\in \IR^d$, $d\geq 1$. That is, the mapping $(z,\omega)\mapsto \theta_z\omega$ is measurable, $\theta_z\circ \theta_x=\theta_{x+z}$ for all $x,z\in \IR^d$ and $\theta_0$ is the identity.
Furthermore, let  $\Q$  be stationary  w.r.t. the flow $\theta$, i.e. $\Q(A)=\Q(\theta_z(A))$ for all $z\in \IR^d$. Finally, we assume  
that the function \begin{align}\label{eq:topology_on_omega}
    \Omega\times \IR^d\to \Omega, \quad(\omega,z)\mapsto \theta_z\omega 
\end{align}
is continuous, where we consider the product topology on $\Omega\times \IR^d$.
We call such probability spaces \textit{admissible}.

In the following a random measure $\xi$ is a locally finite transition kernel from $(\Omega,\mathcal{F})$ to $(\IR^d,\mathcal{B}(\IR^d))$, where locally finite means  that, for  $\IP$-a.e. $\omega \in \Omega$,  the measure $\xi(\omega,\cdot)$ is finite on bounded measurable sets.
A random measure $\xi$ is said to be equivariant if for all $\omega \in \Omega$, $z\in \IR^d$ and $B\in \mathcal{B}(\IR^d)$ it holds that \[
\xi(\omega,B)=\xi(\theta_z\omega,B-z).
\]

We denote by $\Q_{\xi}$ the Palm measure of $\xi$, which is defined by 
\begin{equation}\label{def: Palm}
    \Q_{\xi}(A)=\int\int_{\Lambda_1}\eins_A(\theta_z\omega)\xi(\omega,dz)\Q(d\omega),
\end{equation}
where $\Lambda_n:=[-n/2,n/2]^d$ is the $n$-cube. The refined Campbell theorem yields for every bounded measurable $f:\Omega\times \IR^d\to \IR$ the formula \begin{equation}\label{eq:campbell}
    \IE_{\Q}\left[\int f(\theta_z\omega,z)\xi(\omega,dz)\right]
    =\IE_{\Q_{\xi}}\left[\int f(\omega,z)dz\right].
\end{equation}
A  weighted transport-kernel $T$ (in the following just called kernel) is a kernel from $\Omega\times \IR^d$ to $\IR^d$ such that $T(\omega,z,\cdot)$ is locally finite for all $(\omega,z)\in \Omega\times \IR^d$. It is called invariant if \[
T(\theta_z\omega,u-z,B-z)=T(\omega,u,B),\quad \forall u,z\in \IR^d,\omega\in \Omega, B\in \mathcal{B}(\IR^d).
\]We say that the kernel $T$ balances two random measures $\xi$ and $\eta$, if for all $\omega\in \Omega$ \[
\int T(\omega,z,\cdot)\xi(\omega,dz)=\eta(\omega,\cdot).
\]

Let $\xi$ and $\eta$ be two random measures.
A  coupling $q$ of $\xi$ and $\eta$ is a  transition kernel  from $(\Omega,\mathcal{F})$ to $(\IR^d\times \IR^d,\mathcal{B}(\IR^d\times \IR^d))$ such that for almost every $\omega
\in \Omega$ the measure $q(\omega)$ is a coupling of $\xi(\omega)$ and $\eta(\omega)$, that is \begin{align}\label{eq:semicoupl}
(\pi_1)_{\#}(q(\omega))= \xi(\omega) \text{ and }(\pi_2)_{\#}(q(\omega))=\eta(\omega),
\end{align}
where $\pi_i$ denotes the projection onto the $i$-th coordinate. A coupling $q$ is said to be equivariant if \[
q(\omega,A\times B)=q(\theta_z\omega,(A-z)\times (B-z))\quad \forall \omega\in\Omega, z\in \IR^d, A,B\in \mathcal{B}(\IR^d).
\]
For equivariant random measures $\xi$ and $\eta$ we denote by $\mathsf{cpl}_{e}(\xi,\eta)$ the set of all equivariant couplings of $\xi$ and $\eta$.
Obviously, every invariant $(\xi,\eta)$-balancing kernel $T$ defines a coupling $\q\in \cpl_e(\xi,\eta)$ by \[
\q(\omega)(A\times B)=\int_A \int T(\omega,z,B)\xi(\omega,dz),\quad A,B\in \mathcal{B}(\IR^d)
\]   Conversely, for a given $\q\in \cpl_e(\xi,\eta)$,
let $(\q_z)_{z\in \IR^d}$ be a disintegration of $\q$ w.r.t. the first marginal, i.e. $\q=\q_z\xi(dz)$.
Then \[
T(\omega,z,B)=\q_z(\omega)(B),\quad B\in \mathcal{B}(\IR^d).
\]
defines an invariant balancing kernel.

We are now in a position to  define the cost function between two equivariant random measures $\xi$ and $\eta$.
\begin{defi}\label{def:cost_fct}
Fix $p>1$. For two equivariant random measures $\xi$ and $\eta$ with the same finite intensity set
\begin{align*}
\cost_p(\xi,\eta)&=\inf_{\q\in \cpl_e(\xi,\eta)}\IE_{\Q}\left[\int_{\Lambda_1\times \IR^d}\norm{x-y}^p\q(dx,dy)\right]
&=\inf_{T}\IE_{\Q}\left[\int_{\Lambda_1}
\int \norm{z-u}^pT(\omega,z,du)\xi(\omega,dz)\right],
\end{align*}where the second infimum runs over all invariant $(\xi,\eta)$-balancing kernels $T$.
\end{defi}

\begin{ex}\label{ex:canonical_space}
     Let $\mathcal M(\IR^d)$ be the space of Radon measures  on $\IR^d$ equipped with the topology of vague convergence, which is generated by  the mappings  $\bar{\xi}\mapsto \int fd\bar{\xi}$ for continuous, compactly supported $f\in  C_c(\IR^d)$. Denote the corresponding Borel $\sigma$-algebra by $\mathcal F$. The group $\IR^d$ acts on $\mathcal M(\IR^d)$ by the natural shift operation $\theta_z$, i.e.\ $(\theta_z\bar{\xi})(A)=\bar{\xi}(A+z)$ for any point $z\in\IR^d$, measure $\bar{\xi}\in\mathcal M(\IR^d)$ and any Borel set $ A\subset \IR^d$.
     On $\mathcal M(\IR^d)\times \mathcal M(\IR^d)$ define the diagonal shift by \begin{align*}
         \theta_z(\bar{\xi},\bar{\eta})=(\theta_z\bar{\xi},\theta_z\bar{\eta}),\quad \forall \bar{\xi},\bar{\eta}\in \mathcal M(\IR^d), z\in \IR^d
     \end{align*}
Define $\Cpl_s(\P_0,\P_1)$ to be the set of all couplings $\Q\in \mathcal P\big( \mathcal M(\IR^d)\times \mathcal M(\IR^d)\big)$ between $\P_0$ and $\P_1$ that are stationary in the sense that $\Q\circ(\theta_z,\theta_z)^{-1}=\Q$ for all $z\in \IR^d$. For such a $\Q$ the probability space $(\mathcal M(\IR^d)\times \mathcal M(\IR^d),\mathcal F\otimes \mathcal F,\Q)$ is admissible, where $\mathcal F\otimes \mathcal F$ is the product $\sigma$-algebra. One can then define random measures $\xi$ and $\eta$ as projections, i.e. define \begin{align*}
    \xi(\bar{\xi},\bar{\eta})=\bar{\xi} \text{ and } \eta(\bar{\xi},\bar{\eta})=\bar{\eta},\quad \forall (\bar{\xi},\bar{\eta})\in \mathcal M(\IR^d)\times \mathcal M(\IR^d).
\end{align*}
\end{ex}

\section{Metric structure of $\Omega\times\IR^d$}\label{sec:metric_formulation}
We fix an admissible probability space  $(\Omega,\mathcal F,\Q)$ and two equivariant random measures $\xi$ and $\eta$ defined on $\Omega$.
For a given invariant $(\xi,\eta)$-balancing kernel $T$  we can rewrite the associated cost  in terms of the Palm measure. The invariance of $T$ and the Campbell formula \eqref{eq:campbell} yield \begin{align}\label{eq:cost_in_palm}
\IE_{\Q}\left[\int_{\Lambda_1}
\int \norm{z-u}^pT(\omega,z,du)\xi(\omega,dz)\right]
&= 
\IE_{\Q}\left[\int_{\Lambda_1}
\int \norm{u}^pT(\theta_z\omega,0,du)\xi(\omega,dz)\right]\\
&=\IE_{\Q_{\xi}}\left[\int \norm{u}^pT(\omega,0,du)\right]\nonumber.
\end{align}
We will make use of the following characterization of $(\xi,\eta)$-balancing kernels, proved in \cite[Theorem 4.1]{LaTh09}.
\begin{thm}\label{eq:Charact_transport_palm}
    An invariant weighted transport kernel $T$ is $(\xi,\eta)$-balancing iff \[
    \IE_{\Q_{\xi}}\left[\int f(\theta_z\omega)T(\omega,0,dz)\right]=\IE_{\Q_{\eta}}\left[ f\right]
    \]
    for all measurable $f:\Omega\to [0,\infty)$.
\end{thm}
Combining \eqref{eq:cost_in_palm} and Theorem \ref{eq:Charact_transport_palm} we obtain the following formula for  $\cost(\xi,\eta)$.
\begin{cor}\label{cor:palm_transp}
   We have
   \[
    \cost_p(\xi,\eta)=\inf_{T}\IE_{\Q_{\xi}}\left[\int \norm{z}^pT(\omega,dz)\right],
    \]
where the infimum runs over all locally finite kernels $T$ from $\Omega$ to $\IR^d$ satisfying \begin{equation}\label{eq:transp_prop_kernel}
\IE_{\Q_{\xi}}\left[\int f(\theta_z\omega)T(\omega,dz)\right]=\IE_{\Q_{\eta}}\left[ f\right]
    \end{equation}
    for all measurable $f:\Omega\to [0,\infty)$.
\end{cor}

\begin{rem}\label{rem:construction kernel}
Given a locally finite kernel $T$ from $\Omega$ to $\IR^d$ we can define an invariant weighted transport kernel in the following way. Let $T(\omega,0,\cdot):=T(\omega,\cdot)$ and define for $z\in \IR^d$ \[
T(\omega,z,A):=T(\theta_z\omega,0,A-z),\quad A\in \mathcal{B}(\IR^d).
\]Then $T$ is invariant by construction.
\end{rem}

We define a metric on the product space $\Omega\times \IR^d$.

\begin{defi}
For $\omega_1,\omega_2\in \Omega$ let \begin{equation*}
    \met_{\Omega}(\omega_1,\omega_2)=\inf_{z\in \IR^d:\theta_z\omega_1=\omega_2}\norm{z}.
\end{equation*}
For $(\omega_i,z_i)\in \Omega\times \IR^d$, $i=1,2$ then define \begin{equation}
    \met((\omega_1,z_1),(\omega_2,z_2))=(\norm{z_1-z_2}^2+\met_{\Omega}(\omega_1,\omega_2)^2)^{1/2}.
\end{equation} 
\end{defi}
Then $(\Omega\times \IR^d,\met)$ is an extended metric space and the metric $\met$ is lower semicontinuous.
\begin{lem}\label{lem:met_lsc}
     The metric $\met$ is lower semicontinuous on $(\Omega\times \IR^d)^2$.
\end{lem}
\begin{proof}
Let $(\omega_n,z_n)\xrightarrow{n\to \infty}(\omega,z)$ and $(\omega'_n,z'_n)\xrightarrow{n\to \infty}(\omega',z')$ w.r.t. the  product topology on $\Omega\times \IR^d$. Since we only have to consider the case\[
\liminf_{n\to \infty}\met((\omega_n,z_n),(\omega'_n,z'_n))<\infty,
\]we may assume (by considering subsequences) that $\theta_{u_n}\omega_n=\omega'_n$ with $u_n\xrightarrow{n\to \infty}u\in \IR^d$. It follows from the assumption \eqref{eq:topology_on_omega}  that $\theta_u\omega=\omega'$. Hence,  \[\met^2((\omega,z),(\omega',z'))\leq \norm{u}^2+\norm{z-z'}^2=
\liminf_{n\to \infty}\met^2((\omega_n,z_n),(\omega'_n,z'_n)).
\]
\end{proof}

Let $\wass$ be the set of all probability measure on $\Omega\times \IR^d$. We equip $\wass$ with the Wasserstein distance induced by the metric $\met$, which we denote by $\mathbb W_p$. For $\IP,\bar{\IP}\in \mathcal P(\Omega\times \IR^d)$ it is defined by  \begin{align*}
    \mathbb W_p^p(\IP,\bar{\IP})=\inf_{\Q\in \Cpl(\IP,\bar{\IP})}\IE_{\Q}[\met^p].
\end{align*} 

Then

\begin{lem}\label{lem:wass_lsc}
The metric $\mathbb{W}_p$ is lower semicontinuous w.r.t. to the weak convergence on $\mathcal{P}(\Omega\times \IR^d)$.
\end{lem}
\begin{proof}
This follows from the lower semicontinuity of $\met$ and \cite[Lemma 4.3]{Villani}.
\end{proof}

The cost function can  be rewritten in terms of this metric. 
\begin{prop}\label{prop:wass_dist_palm}
    We have \begin{align}
        \cost_p(\xi,\eta)=\inf_{\IP_0,\IP_1}\inf_{\U\in \Cpl(\IP_0,\IP_1)}\IE_{\U}\left[\met^p\right]=\inf_{\IP_0,\IP_1}\mathbb W_p^p(\IP_0,\IP_1),
    \end{align}
    where the first infimum runs over all probability measures $\IP_i$, $i=0,1$, on $\Omega\times \IR^d$ such that $(\proj_{\Omega})_{\#}\IP_0=\Q_{\xi}$ and $(\proj_{\Omega})_{\#}\IP_1=\Q_{\eta}$.
\end{prop}
\begin{proof}
    Let $\IP_i$, $i=0,1$, and $\U\in \Cpl(\IP_0,\IP_1)$ be given such that $\IE_{\U}[\met^p]<\infty$.
    The measure $\U$ induces a measure $\tilde{\U}$ on $\Omega^2$ by setting for $f:\Omega^2\to \IR$ bounded and measurable \[
    \int f(\omega_1,\omega_2)d\tilde{\U}(\omega_1,\omega_2)= \int f(\omega_1,\omega_2)d\U((\omega_1,z_1),(\omega_2,z_2)).
    \] 
    In particular \[
    \IE_{\tilde{\U}}[\met_{\Omega}^p]\leq \IE_{\U}[\met^p]<\infty,
    \] which implies that $\met_{\Omega}^p$ is 
    $\tilde{\U}$-a.e. finite. Define the map \begin{align*}
        F:\{(\omega_1,\omega_2)\in \Omega^2:\met_{\Omega}(\omega_1,\omega_2)<\infty\}\to \Omega\times \IR^d, 
    \end{align*}
    by $F(\omega_1,\omega_2)=(\omega_1,\min\{z\in \IR^d:\theta_z\omega_1=\omega_2\})$, where we use the the following ordering on $\IR^d$:
    \begin{align*}
        x< y \iff \norm{x}\leq \norm{y}
    \end{align*}
    and in the case $\norm{x}=\norm{y}$ we use the lexicographic ordering as a tie-breaker.
    Disintegrating the pushforward $F_{\#}\tilde{\U}$ with respect to the first coordinate, we obtain a kernel $T$, which satisfies for bounded and measurable $f:\Omega\to \IR$ \begin{align*}
        \IE_{\Q_{\xi}}\left[\int f(\theta_z\omega)T(\omega,dz)\right]=
        \int f(\theta_z\omega)dF_{\#}\tilde{\U}(\omega,z)=
        \int f(\omega_2)d\tilde{\U}(\omega_1,\omega_2)=
         \IE_{\Q_{\eta}}\left[ f\right].
    \end{align*}
    Hence, the kernel $T$ satisfies \eqref{eq:transp_prop_kernel}. Furthermore, 
    \begin{align}\label{eq:cost_kernel}
        \IE_{\Q_{\xi}}\left[\int \norm{z}^pT(\omega,dz)\right]=
        \IE_{\tilde{\U}}[\met_{\Omega}^p]\leq \IE_{\U}[\met^p]
    \end{align}
Hence, Corollary \ref{cor:palm_transp} and  equation \eqref{eq:cost_kernel} show that \[
\cost_p(\xi,\eta)\leq \inf_{\IP_0,\IP_1}\inf_{\U\in \Cpl(\IP_0,\IP_1)}\IE_{\U}[\met^p].
\]

Let now $T$ be a kernel from $\Omega$ to $\IR^d$ satisfying property \eqref{eq:transp_prop_kernel}. We define the measure $\U$ on $(\Omega \times \IR^d)^2$ by \begin{align*}
    \IE_{U}\left[f\right]=\IE_{\Q_{\xi}}\left[
    \int f(\omega,0,\theta_z\omega,0)T(\omega,dz)\right],
\end{align*} 
for all  measurable functions $f:(\Omega \times \IR^d)^2\to [0,\infty)$. Since $T$ satisfies  \eqref{eq:Charact_transport_palm}, we have $\U\in \Cpl(\Q_{\xi}\otimes \delta_0,\Q_{\eta}\otimes\delta_0)$. Furthermore, \begin{align*}
    \IE_{\U}\left[\met^p\right]&=\IE_{\Q_{\xi}}\left[\int \met^p(\omega,0,\theta_z\omega,0)T(\omega,dz)\right]\\
    &\geq \IE_{\Q_{\xi}}\left[\int \norm{z}^pT(\omega,dz)\right]
\end{align*}
and hence
\[
\cost_p(\xi,\eta)\geq \inf_{\IP_0,\IP_1}\inf_{\U\in \Cpl(\IP_0,\IP_1)}\IE_{\U}[\met^p].
\]
\end{proof}

\section{Continuity Equation}\label{sec:cont_equation}

In this section, we closely follow the general strategy of \cite[Chapter 8]{AGS08} adapted to our setup. 
Unless explicitly stated otherwise, we fix an admissible probability space  $(\Omega,\mathcal F,\Q)$. 
\begin{defi}
    For a function $\Phi:\Omega\to \IR$ and $\omega\in \Omega$ we denote by $\nabla_{\Omega}\Phi(\omega)$ the unique vector, if it exists, such that the following holds: Let  $\varepsilon>0$ and $c:(-\varepsilon,\varepsilon)\to \IR^d$ be a curve with $c(0)=0$, which is differentiable at $t=0$, then \begin{equation*}
        \frac{d}{dt}\Bigr|_{t=0}\Phi(\theta_{c(t)}\omega)=\inner{c'(0)}{\nabla_{\Omega} \Phi(\omega)},
    \end{equation*}  
    where $\inner{\cdot}{\cdot}$ denotes the Euclidean inner product on $\IR^d$.
\end{defi}

\begin{ex}\label{example}
Let $\rho_{\varepsilon}(z,u)=(2\pi \varepsilon)^{-d}\exp({-\norm{(z,u)}^2/2\varepsilon})$ and   $f:\Omega \times \IR^d\to \IR$ be continuous (w.r.t. the product topology) and bounded.  Set\[
 f_{\varepsilon}(\omega,x)=\int_{\IR^d}\int_{\IR^d} \rho_{\varepsilon}(z,u) f(\theta_z\omega,x-u)dzdu.
 \]
 For $r,v\in \IR^d$ we then have \begin{align*}
        f_{\varepsilon}(\theta_r\omega,x+v)=\int_{\IR^d}\int_{\IR^d} \rho_{\varepsilon}(z-r,u+v) f(\theta_z\omega,x-u)dzdu
    \end{align*}
    and hence \begin{align*}
        \left(-\nabla_{\Omega} f_{\varepsilon}(\omega,x),\nabla_{\IR^{d}} f_{\varepsilon}(\omega,x)\right)^T=\int_{\IR^d}\int_{\IR^d} \nabla\rho_{\varepsilon}(z,u) f(\theta_z\omega,x-u)dzdu\in \IR^{2d}.
    \end{align*}
    In particular, the gradient of $f_{\varepsilon}$ is uniformly bounded and continuous w.r.t. the metric $\met$.
\end{ex}

To be able to formulate the continuity equation we need to introduce a suitable space of  test functions.
\begin{defi}\label{def:test_fct}
Let $\test$ be the set of all functions  $\Phi:[0,T]\times \Omega\times \IR^d\to \IR$ with the following properties\begin{enumerate}[i)]
\item The map $\Phi$ is bounded and  measurable w.r.t. the product topology on $[0,T]\times \Omega\times \IR^d$.
    \item There exist compact sets $K_1\subset (0,T), K_3\subset \IR^d$ such that $\Phi_t(x,\omega)=0$ if $t\notin K_1$ or $x\notin K_3$.
    \item The partial derivative $\partial_t \Phi$ and the gradients w.r.t $\Omega$ and $\IR^d$ exist  and are bounded uniformly, i.e. \begin{align}\label{eq:unifbound_dev}
    \sup_{(t,x,\omega)\in (0,T)\times \IR^d\times \Omega}\lvert \partial_t \Phi_t(x,\omega)\rvert+\lVert \nabla_{\Omega}\Phi_t(x,\omega)\rVert+\lVert \nabla_{\IR^d}\Phi_t(x,\omega)\rVert<\infty.
    \end{align}
    \item For fixed $0\leq t\leq T$ the gradient $\nabla \Phi_t$ is continuous with respect to the metric $\met$.
\end{enumerate}
\end{defi}
For $\Phi_t\in \test$ we use the notation $\nabla\Phi_t=(\nabla_{\Omega}\Phi_t,\nabla_{\IR^d}\Phi_t)\in \IR^{2d}$.
The space of test functions is rich in the following sense.

\begin{lem}
    Let $(\IP_t)_{0\leq t\leq T}$ be  a weakly continuous curve of finite signed  measures on $\Omega\times \IR^d$ and assume that for all nonnegative functions $\Psi \in \test$ \begin{align*}
        \int_0^T\int_{\Omega\times \IR^d} \Psi_t(\omega,x) d\IP_t(\omega,x) dt\leq 0.
    \end{align*}
    Then $\IP_t$ is a negative measure for all $0\leq t\leq T$.
\end{lem}
\begin{proof}
    Let $f:\Omega \times \IR^d\to [0,\infty)$ be continuous (w.r.t. the product topology) and bounded and let $\rho_{\varepsilon}$ be defined by $\rho_{\varepsilon}(z,u)=(2\pi \varepsilon)^{-d}\exp({-(\norm{z}^2+\norm{u}^2)/\varepsilon})$, $z,u\in \IR^d$. Define \begin{align*}
        f_{\varepsilon}(\omega,x)=\int_{\IR^d}\int_{\IR^d} \rho_{\varepsilon}(z,u) f(\theta_z\omega,x-u)dzdu.
    \end{align*}
    Then $f_{\varepsilon}(\omega,x)\xrightarrow{\varepsilon\to \infty}f(\omega,x)$ for all $\omega\in \Omega$ and $x\in \IR^d$. Let $g\in C_c^{\infty}(\IR^d)$ and $h\in C^{\infty}_c((0,T))$ be nonnegative. Set $\Psi^{\varepsilon}_t(\omega,x)=h(t)g(x)f_{\varepsilon}(\omega,x)$. Then, by Example \ref{example},  $\Psi^{\varepsilon}\in \test$ and hence\begin{align}\label{eq:dense_test}
        0\geq  \int_0^T\int_{\Omega\times \IR^d} \Psi^{\varepsilon}_t(\omega,x) d\IP_t(\omega,x) dt \xrightarrow{\varepsilon\to 0} \int_0^T\int_{\Omega\times \IR^d} h(t)g(x)f(\omega,x) d\IP_t(\omega,x) dt.
    \end{align}
    by dominated convergence.
    Note that by the weak continuity of the curve $(\IP_t)_{0\leq t\leq 1}$, the function \[
    [0,1]\ni t\mapsto \int_{\Omega\times \IR^d} g(x)f(\omega,x) d\IP_t(\omega,x)
    \]
    is continuous. Hence, \eqref{eq:dense_test} implies that 
\begin{align*}
    0\geq \int_{\Omega\times \IR^d} g(x)f(\omega,x) d\IP_t(\omega,x), \quad \forall g\in C_c^{\infty}(\IR^d).
\end{align*}
    Considering  positive functions $g_n\in C_c^{\infty}(\IR^d)$ with $g_n(x)\nearrow 1$ as $n\to \infty$, $x\in \IR^d$, we  obtain 
\begin{align*}
       0\geq  \int_{\Omega\times \IR^d} f(\omega,x) d\IP_t(\omega,x),
    \end{align*}
which proves the claim.
\end{proof}

\begin{defi}\label{def:cont_equat}
Let $(\IP_t)_{0\leq t\leq T}$ be  a weakly continuous curve of finite signed  measures on $\Omega\times \IR^d$ and 
$V:[0,T]\times \Omega\times \IR^d\to \IR^d\times \IR^d$ a Borel vector field. The pair $((\IP_t)_{t\in [0,T]},(V_t)_{t\in [0,T]})$ is said to satisfy the continuity equation if for all functions $\Phi\in \test$
 \begin{equation}\label{eq:CE}
    \int_0^T\int_{\Omega\times \IR^d} \partial_t \Phi + \inner{V_t} {\nabla \Phi_t} d\IP_t dt=0\tag{CE}. 
\end{equation}
\end{defi}

The following  comparison principle holds.

\begin{prop}\label{prop:comparison}
   Let $((\IP_t)_{t\in [0,T]},(V_t)_{t\in [0,T]})$ satisfy \eqref{eq:CE}  with $\IP_0\leq0$ and \begin{align*}
        \int_0^T\int \norm{V_t}d\abs{\IP_t}dt<\infty,
    \end{align*}
    where $\abs{\IP_t}$ is the total variation of $\IP_t$.
    Furthermore, assume that for all compact sets $B\subset \IR^d$ \begin{align}\label{eq:prop_assump_sup_lip}
        \int_0^T\abs{\IP_t}(\Omega \times B )+\sup_{\Omega\times B}\norm{V_t}+Lip(V_t,\Omega\times B)dt<\infty,
    \end{align}
    where $Lip(V_t,\Omega\times B)$ is the Lipschitz constant of $V_t:\Omega\times B\to \IR^{d}\times \IR^d$ w.r.t. the metric $\met$. Then $\IP_t\leq 0$ for any $t\in [0,T]$.
\end{prop}

\begin{proof}
 We follow closely the strategy of \cite[Proposition 8.1.7]{AGS08} using the method of characteristics. Due to the invariance in our setting we need to slightly modify the argument.

Step 1: Mollification of the vector field $V$. For $R>0$ fix a function $f_R\in C^{\infty}_c(\IR^d)$ such that $f_R=1$ on $B_{2R}(0)$ and $f_R=0$ on $B_{3R}(0)$. Define for  $(\omega,x)\in \Omega\times \IR^d$ and $0\leq t\leq T$\begin{align}\label{eq:def_V'}V_t'(\omega,x)=f_R(x)V_t(\omega,x).\end{align} Let $y\in \IR^d$, then if $x,x+y\in B_{3R}(0)$ by the triangle inequality \begin{align*}
        \norm{V'_t(\omega,x+y)-V_t'(\omega,x)}
        \leq \norm{y} \norm{\nabla f_R}_{\infty} \sup_{\Omega\times B_{3R}(0)}\norm{V_t}+ \norm{f_R}_{\infty} \norm{y} Lip(V_t,\Omega\times B_{3R}(0))
    \end{align*}
If $x\in B_{3R}(0)$ and $x+y\notin B_{3R}(0)$, then \[
\norm{V'_t(\omega,x)}=\norm{f_R(x)V_t(\omega,x)-f_R(x+y)V_t(\omega,x)}\leq \norm{\nabla f_R}_{\infty}\norm{y} \sup_{\Omega\times B_{3R}(0)}\norm{V_t}.
\]
Hence, for fixed $\omega\in \Omega$ the function $V'_t(\omega,x)$ is Lipschitz on $\IR^d$ with  $R$-dependent constant $\sqrt{2}L_R$, given by \begin{equation}\label{eq:def_L_R}
    L_R= (\norm{f_R}_{\infty}+\norm{\nabla f_R}_{\infty})(\sup_{\Omega\times B_{3R}(0)}\norm{V_t}+Lip(V_t,\Omega\times B_{3R}(0)))
\end{equation}
 For $\omega\in \Omega$ and $x,y,z\in \IR^d$\begin{align*}
     \norm{V'_t(\omega,x)-V'_t(\theta_z\omega,x+y)}&\leq
     \norm{V'_t(\omega,x)-V'_t(\theta_z\omega,x)}+
      \norm{V'_t(\theta_z\omega,x)-V'_t(\theta_z\omega,x+y)}\\
      &\leq \norm{f_R}_{\infty} \met_{\Omega}(\omega,\theta_z\omega) Lip(V_t,\Omega\times B_{3R}(0))+
      L_R \norm{y}.
 \end{align*}    
Hence, $V'_t$ is Lipschitz w.r.t. the metric $\met$ on $\Omega\times \IR^d$ with constant $L_R$. Moreover, we have \begin{align}\label{eq:Lip_const_R}
    \sup_{\Omega\times \IR^d}\norm{V'_t}\leq 
    \sup_{\Omega\times B_{3R}(0)}\norm{V_t}.
\end{align}
For $t\notin [0,T]$ let $V'_t=0$. Consider the family of mollifiers $(\rho_{\varepsilon})_{\varepsilon>0}$, defined by \[\rho_{\varepsilon}(t,x,y)=(2\pi \varepsilon)^{-(2d+1)/2}\exp({-(t^2+\norm{x}^2+\norm{y}^2)/2\varepsilon}),\quad  \forall (t,x,y)\in \IR\times \IR^d\times \IR^d.\]

We set \begin{align*}
    V^{\epsilon}_t(\omega,x)=\int_{\IR\times \IR^d\times \IR^d} \rho_{\varepsilon}(s,r,u)V'_{t+s}(\theta_r\omega,x+u)dsdrdu.
\end{align*}
Hence, by \eqref{eq:Lip_const_R} and  assumption \eqref{eq:prop_assump_sup_lip} \begin{align}\label{eq:integr_sup_lim}
    S:=\sup_{\epsilon>0}\int_0^T \sup_{\Omega\times \IR^d} \norm{V^{\epsilon}_t}+ Lip(V^{\epsilon}_t,\Omega\times \IR^d)<\infty.
\end{align}
Moreover, \begin{align*}
    \norm{V^{\epsilon}_t(\omega,x)}&\leq \int_{\IR\times \IR^d\times \IR^d} \rho_{\varepsilon}(s,r,u)\sup_{\Omega\times \IR^d}\norm{V'_{t+s}}dsdrdu\\
    &\leq(2\pi \varepsilon)^{-(2d+1)/2} \int_{\IR\times \IR^d\times \IR^d} \exp({-(\norm{r}^2+\norm{u}^2)/2\varepsilon})\sup_{\Omega\times \IR^d}\norm{V'_{t+s}}dsdrdu\\
    &=(2\pi \varepsilon)^{-(2d+1)/2} \int_{\IR\times \IR^d} \exp({-(\norm{r}^2+\norm{u}^2)/2\varepsilon})\left(\int_0^T\sup_{\Omega\times \IR^d}\norm{V'_{s}}ds\right)drdu\\
&=C_{\varepsilon}\int_0^T\sup_{\Omega\times \IR^d}\norm{V'_{s}}ds
\end{align*}
for some $\varepsilon$ dependent constant $C_{\varepsilon}$. Then assumption \eqref{eq:prop_assump_sup_lip} and \eqref{eq:Lip_const_R} show that 
\begin{align}\label{eq:unif_bound_vepislon}
\sup_{[0,T]\times \Omega\times \IR^d} \norm{V^{\epsilon}}<\infty.
\end{align}

For $\omega\in \Omega$ and $t\in \IR$ define $G_t^{\omega,\epsilon}:\IR^{2d}\to \IR^{2d}$ by \begin{align}\label{eq:def_G}
G_t^{\omega,\epsilon}(x,y)=V^{\epsilon}_t(\theta_x\omega,y).\end{align} It follows immediately that \begin{align*}
    \sup_{\IR^{2d}}\norm{G^{\omega,\epsilon}_t}\leq \sup_{\Omega\times \IR^d} V^{\epsilon}_t
\end{align*}

and \begin{align*}
    Lip(G^{\omega,\varepsilon}_t,\IR^{2d})\leq  Lip(V^{\epsilon}_t,\Omega\times \IR^d).
\end{align*}
Thus, by \eqref{eq:integr_sup_lim} for any $\omega\in \Omega$\begin{align}\label{eq:abcd}
   \sup_{\epsilon>0}\int_0^T\sup_{\IR^{2d}}\norm{G^{\omega,\epsilon}_t} + 
   Lip(G^{\omega,\varepsilon}_t,\IR^{2d})dt
   \leq S< \infty
\end{align} 

Step 2: Solving the random ODE. For any $\omega\in \Omega$, $s\in [0,T]$, $\varepsilon>0$ and $(x,y)\in \IR^{2d}$ , we apply \cite[Lemma 8.1.4]{AGS08} to obtain a unique solution of the ODE \begin{align}\label{eq:ode_IR^d}
    \partial_tX^{\omega,\epsilon}_t(x,y,s)=
G^{\omega,\varepsilon}_t(X^{\omega,\epsilon}_t(x,y,s)), \quad X^{\omega,\epsilon}_s(x,y,s)=(x,y),
\end{align} 
 defined on the interval $[0,T]$. Furthermore, combining \cite[Lemma 8.1.4]{AGS08} with \eqref{eq:abcd} yields that for any $\omega\in \Omega$ and $s\in [0,T]$
 \begin{align*}
     \int_0^T\sup_{(x,y)\in \IR^{2d}}\norm{\partial_t X^{\omega,\epsilon}_t(x,y,s)}dt\leq \int_0^T\sup_{\IR^{2d}}\norm{G^{\omega,\epsilon}_t} + 
   Lip(G^{\omega,\varepsilon}_t,\IR^{2d})dt\leq S
     \end{align*} 
     and 
     \begin{align}\label{eq:lip_bounds_IR2}
     \sup_{s,t\in [0,T]} Lip(X^{\omega,\epsilon}_t(\cdot,\cdot,s),\IR^{2d})\leq \exp\left(\int_0^T\sup_{\IR^{2d}}\norm{G^{\omega,\epsilon}_t} + 
   Lip(G^{\omega,\varepsilon}_t,\IR^{2d})dt\right) \leq e^S.
 \end{align}
 Uniqueness of solutions implies that for $\omega\in \Omega$, $x,y\in \IR^d$ and $0\leq s\leq t\leq T$ we have \begin{align}\label{eq:uniq_sol}
      X^{\omega,\varepsilon}_t(X_s^{\omega,\varepsilon}(x,y,0),s)=X^{\omega,\varepsilon}_t(x,y,0).
 \end{align}
We  use the following notation for the two $\IR^d$ components \begin{align*}X^{\omega,\epsilon}_t(x,y,s)=(U^{\omega,\epsilon}_t(x,y,s),W^{\omega,\epsilon}_t(x,y,s)).\end{align*}
By Definition \eqref{eq:def_G}, for $a\in \IR^d$   \[
G^{\theta_a\omega,\varepsilon}_t(x-a,y)=
V_t^{\epsilon}(\theta_{x-a}(\theta_a\omega),y)=G^{\omega,\epsilon}_t(x,y).
\]
% and thus 
% \begin{align*}
%     G^{\theta_a\omega,\epsilon}_t(x-a,y)=G^{\omega,\epsilon}_t(x,y).
% \end{align*}
Hence,  uniqueness of the solutions of \eqref{eq:ode_IR^d} implies  that \begin{align}\label{eq:displ_sol_IRd}
    X_t^{\theta_a\omega,\varepsilon}(x-a,y,s)=X_t^{\omega,\varepsilon}(x,y,s)-(a,0).
\end{align}

Now define \begin{align*}
    Z^{\epsilon}_t(\omega,y,s)=(\theta_{U^{\omega,\epsilon}_t(0,y,s)}\omega,W^{\omega,\epsilon}_t(0,y,s)).
\end{align*}
Then we have \begin{align}\label{eq:calc_charact_omega}
    Z^{\varepsilon}_t(Z_s^{\varepsilon}(\omega,y,0),s)=Z^{\varepsilon}_t(\omega,y,0).
\end{align}
Indeed, by definition of $Z_{\varepsilon}$
\begin{align}\label{eq:ZZ3}
     Z^{\varepsilon}_t(Z_s^{\varepsilon}(\omega,y,0),s)&=(\theta_{U^{\bar\omega,\epsilon}_t(0,\bar y,s)}\bar\omega,W^{\bar\omega,\epsilon}_t(0,\bar y,s)),
\end{align}
with $\bar\omega=\theta_{U^{\omega,\epsilon}_s(0,y,0)}\omega$ and $\bar y=W^{\omega,\epsilon}_s(0,y,0)$.
By \eqref{eq:displ_sol_IRd} \begin{align*}
    U^{\bar\omega,\epsilon}_t(0,\bar y,s)=U^{\omega,\epsilon}_t(U^{\omega,\epsilon}_s(0,y,0),\bar y,s)-U^{\omega,\epsilon}_s(0,y,0)=U^{\omega,\epsilon}_t(X^{\omega,\epsilon}_s(0,y,0),s)-U^{\omega,\epsilon}_s(0,y,0).
\end{align*}
Equation \eqref{eq:uniq_sol} shows that $U^{\omega,\epsilon}_t(X^{\omega,\epsilon}_s(0,y,0),s)=U^{\omega,\varepsilon}_t(0,y,0)$ and hence  \begin{align}\label{eq:ZZ1}
\theta_{U^{\bar\omega,\epsilon}_t(0,\bar y,s)}\bar\omega=\theta_{U^{\omega,\varepsilon}_t(0,y,0)-U^{\omega,\epsilon}_s(0,y,0)}\theta_{U^{\omega,\epsilon}_s(0,y,0)}\omega=
\theta_{U^{\omega,\varepsilon}_t(0,y,0)}\omega.
\end{align}
Next we show that $W^{\bar\omega,\epsilon}_t(0,\bar y,s)=W^{\omega,\epsilon}_t(0, y,0)$.
Note that by \eqref{eq:displ_sol_IRd}\begin{align*}
    W^{\bar\omega,\epsilon}_t(0,\bar y,s)=W^{\omega,\epsilon}_t(U^{\omega,\epsilon}_s(0,y,0),\bar y,s)=W^{\omega,\epsilon}_t(X^{\omega,\epsilon}_s(0,y,0),s)
\end{align*}
and hence \eqref{eq:uniq_sol} implies that \begin{align}\label{eq:ZZ2}
 W^{\bar\omega,\epsilon}_t(0,\bar y,s)= W^{\omega,\epsilon}_t(0, y,0).
\end{align}
Plugging \eqref{eq:ZZ1} and \eqref{eq:ZZ2} into \eqref{eq:ZZ3} implies \eqref{eq:calc_charact_omega}. 

\medskip

Step 3: Constructing a good test function admissible for \eqref{eq:CE}. Fix a test function $\Psi\in \test$ with $0\leq \Psi\leq 1$ and define \begin{align*}
    \Phi^{\varepsilon}_t(\omega,y)=-\int_t^T\Psi(Z_u^{\epsilon}(\omega,y,t),u)du.
\end{align*}
Then the identity \eqref{eq:calc_charact_omega} yields \begin{align*}
    \Phi^{\varepsilon}_t(\omega,Z^{\varepsilon}_t(\omega,y,0))&=-\int_t^T\Psi(Z_u^{\epsilon}(\omega,y,0),u)du.
\end{align*}
 Differentiation with respect to $t$ yields \begin{align*}
 \Psi_t(Z_t^{\epsilon}(\omega,y,0))&=
\partial_t\Phi^{\varepsilon}_t(Z^{\varepsilon}_t(\omega,y,0))\\
&+ \inner{\nabla_{\Omega} \Phi^{\varepsilon}_t(Z^{\varepsilon}_t(\omega,y,0))}{\partial_tU^{\omega,\epsilon}_t(0,y,0)}\\
&+\inner{\nabla_{\IR^d} \Phi^{\varepsilon}_t(Z^{\varepsilon}_t(\omega,y,0))}{\partial_tW^{\omega,\epsilon}_t(0,y,0)}\\
&=\partial_t\Phi^{\varepsilon}_t(Z^{\varepsilon}_t(\omega,y,0))\\
&+\inner{\nabla \Phi^{\varepsilon}_t(Z^{\varepsilon}_t(\omega,y,0))}{V^{\epsilon}_t(Z^{\varepsilon}_t(\omega,y,0))},
 \end{align*}
where in the last line we used \eqref{eq:ode_IR^d} and the definition of $G_t^{\omega,\varepsilon}$, see \eqref{eq:def_G}.
 Since this holds for all $0<t<T$ and $(\omega,y)\in \Omega\times \IR^d$, we obtain \begin{align}\label{eq:transport_psi}
   \partial_t\Phi^{\varepsilon}_t+\inner{\nabla\Phi_t^{\varepsilon}}{V_t^{\varepsilon}}=\Psi_t   
 \end{align}
 on $(0,T)\times \Omega\times \IR^d$ (cf. \cite[Remark 8.1.5]{AGS08}). Furthermore, $\Phi^{\varepsilon}_T=0$, $0\geq\Phi^{\varepsilon}\geq -T$. 
 
We claim that for suitable functions $\eta_r,\chi_R$ we have $\eta_r(t)\chi_R(x)\Phi^\varepsilon_t(\omega,y)\in\mathcal D$. 
To this end, we need to control the gradients of the function $\Phi^{\varepsilon}_t$.
Note that for $a\in \IR^d$ \begin{align}\label{eq:abc}
    Z^{\epsilon}_u(\theta_a\omega,y,t)&=(\theta_{U^{\omega,\varepsilon}_u(a,y,t)}\omega,W^{\omega,\varepsilon}_u(a,y,t)).
\end{align}
This follows from the definition of $Z^{\varepsilon}$ and from \eqref{eq:displ_sol_IRd}\begin{align*}
    Z^{\epsilon}_u(\theta_a\omega,y,t)&=(\theta_{U^{\theta_a\omega,\epsilon}_u(0,y,t)}\theta_a\omega,W^{\theta_a\omega,\epsilon}_u(0,y,t))\\
    &=(\theta_{U^{\theta_a\omega,\epsilon}_u(0,y,t)+a}\omega,W^{\theta_a\omega,\epsilon}_u(0,y,t))\\
        &=(\theta_{U^{\omega,\epsilon}_u(a,y,t)-a+a}\omega,W^{\omega,\epsilon}_u(a,y,t))\\
         &=(\theta_{U^{\omega,\epsilon}_u(a,y,t)}\omega,W^{\omega,\epsilon}_u(a,y,t)).
\end{align*}
In particular, \eqref{eq:abc} and \eqref{eq:lip_bounds_IR2} imply that for $0\leq t,u\leq T$ the map \begin{align}\label{eq:cont_Z}
    Z^{\varepsilon}_u(\cdot,\cdot,t):\Omega\times \IR^d\to \Omega\times \IR^d
\end{align}
is continuous w.r.t. the metric $\met$.
We have 
\begin{align*}
\Psi(Z^{\varepsilon}_u(\theta_a\omega,y,t),u)
=\Psi(\theta_{U^{\omega,\varepsilon}_u(a,y,t)}\omega,W^{\omega,\varepsilon}_u(a,y,t),u).
\end{align*}
and hence
\begin{align}\label{eq:formula1}
&\nabla_{\Omega}\Psi(Z^{\varepsilon}_u(\omega,y,t),u) \\
    &=(D_x\Bigr|_{x=0}U_u^{\omega,\epsilon}(x,y,t))^T(\nabla_{\Omega}\Psi)(Z^{\varepsilon}_u(\omega,y,t),u)\nonumber\\
    &+(D_x\Bigr|_{x=0}W_u^{\omega,\epsilon}(x,y,t))^T(\nabla_{\IR^d}\Psi)(Z^{\varepsilon}_u(\omega,y,t),u)\nonumber,
\end{align}
where $(D_x\Bigr|_{x=0}U_u^{\omega,\epsilon}(x,y,t))^T$ is the transposed Jacobian of $U_u^{\omega,\epsilon}(\cdot,y,t)$ evaluated at $0$ and similarly for the other term. 
Here we used that $\nabla \Psi_u$ is continuous w.r.t. the metric $\met$.
We claim that for fixed $0\leq t\leq u\leq T$ the Jacobian $D_x\Bigr|_{x=0}U_u^{\omega,\epsilon}(x,y,t)$ is continuous in $\omega\in \Omega$ and $y\in \IR^d$ with respect to the metric $\met$.
Fix $0\leq t\leq u\leq T$ and let $z_n\xrightarrow{n\to \infty}0$ and $y_n\xrightarrow{n\to \infty}y$. We have to show \[
D_x\Bigr|_{x=0}U_u^{\theta_{z_n}\omega,\epsilon}(x,y_n,t)\xrightarrow{n\to \infty}D_x\Bigr|_{x=0}U_u^{\omega,\epsilon}(x,y,t).
\]
Equation \eqref{eq:displ_sol_IRd} implies that \[
U_u^{\theta_{z_n}\omega,\epsilon}(x,y_n,t)=U_u^{\omega,\epsilon}(x+z_n,y_n,t)-z_n
\]
and hence it suffices to show 
\begin{align}\label{eq:conv_deriv}
D_x\Bigr|_{x=0}U_u^{\omega,\epsilon}(x+z_n,y_n,t)\xrightarrow{n\to \infty}D_x\Bigr|_{x=0}U_u^{\omega,\epsilon}(x,y,t).
\end{align}
Note that $U_u^{\omega,\epsilon}(x+z_n,y_n,t)$ and $U_u^{\omega,\epsilon}(x,y,t)$ are the first components of solutions to an ODE with smooth vector field $G^{\omega,\varepsilon}$ (cf. \eqref{eq:ode_IR^d}) and initial values $(x+z_n,y_n)$ and $(x,y)$ respectively. Standard facts about dependence of ODEs  on initial conditions (cf. \cite[Theorem 4.1]{Hartman2002OrdinaryDE}) then imply \eqref{eq:conv_deriv}. 
The same reasoning applies to the term $D_x\Bigr|_{x=0}W_u^{\omega,\epsilon}(x,y,t)$. Hence, \eqref{eq:cont_Z} implies that $\nabla_{\Omega}\Psi(Z^{\varepsilon}_u(\omega,y,t),u)$ is continuous with respect to $\met$.

With the same reasoning \begin{align}\label{eq:formula2}
&\nabla_{\IR^d}\Psi(Z^{\varepsilon}_u(\omega,y,t),u) \\
    &=(D_yU_u^{\omega,\epsilon}(0,y,t))^T(\nabla_{\Omega}\Psi)(Z^{\varepsilon}_u(\omega,y,t),u)\nonumber\\
    &+(D_yW_u^{\omega,\epsilon}(0,y,t))^T(\nabla_{\IR^d}\Psi)(Z^{\varepsilon}_u(\omega,y,t),u)\nonumber
\end{align}
and $\nabla_{\IR^d}\Psi(Z^{\varepsilon}_u(\omega,y,t),u)$ is continuous with respect to $\met$.
Thus $\nabla\Psi(Z^{\varepsilon}_u(\omega,y,t),u)$ is continuous with respect to $\met$ and, since $\Psi\in \test$, the   formulas \eqref{eq:formula1}, \eqref{eq:formula2} and  \eqref{eq:lip_bounds_IR2} imply that  
\begin{align*}
   \sup_{\epsilon>0} \sup_{[0,T]\times \Omega\times \IR^d}\norm{\nabla_{\Omega}\Psi(Z^{\varepsilon}_u(\omega,y,t),u)}+\norm{\nabla_{\IR^d}\Psi(Z^{\varepsilon}_u(\omega,y,t),u)}<\infty.
\end{align*}Since \begin{align*}
    \nabla \Phi^{\varepsilon}_t(\omega,y)
    =-\int_t^T\nabla \Psi(Z^{\epsilon}_u(\omega,y,t),u)du,
\end{align*}
 we obtain that 
 \begin{align}\label{eq:unif_bound_grad}
   \sup_{\epsilon>0} \sup_{[0,T]\times \Omega\times \IR^d}\norm{\nabla\Phi^{\varepsilon}_t(\omega,y))}<\infty,
\end{align}
where again $\nabla \Phi^{\varepsilon}_t=(\nabla_{\Omega} \Phi^{\varepsilon}_t,\nabla_{\IR^d} \Phi^{\varepsilon}_t)$.
Furthermore, since for every $0\leq u\leq T$ the function $\nabla \Psi(Z^{\epsilon}_u(\omega,y,t),u)$ is continuous with respect to the metric $\met$, also $\nabla   \Phi^{\varepsilon}_t$ is  continuous w.r.t.\ the metric $\met$.
Next, we turn to the time derivative of $\Phi^{\varepsilon}_t(\omega,y)$. We have  
\begin{align*}
    \partial_t\Phi^{\varepsilon}_t(\omega,y)&=
    \Psi(Z^{\epsilon}_t(\omega,y,t),t)
    -\int_t^T\partial_t \Psi(Z^{\epsilon}_u(\omega,y,t),u)du\\
    &=
    \Psi(\omega,y,t)
    -\int_t^T\partial_t \Psi(\theta_{U_u^{\omega,\epsilon}(0,y,t)}\omega,W_u^{\omega,\epsilon}(0,y,t),u)du\\
    &=
    \Psi(\omega,y,t)
    +\int_t^T\inner{(\nabla\Psi)(\theta_{U_u^{\omega,\epsilon}(0,y,t)}\omega,W_u^{\omega,\epsilon}(0,y,t),u)}{G^{\omega,\epsilon}_t(U_u^{\omega,\epsilon}(0,y,t),W_u^{\omega,\epsilon}(0,y,t))}du\\
    &=
    \Psi(\omega,y,t)
    +\int_t^T\inner{(\nabla\Psi)(Z^{\epsilon}_u(\omega,y,t),u)}{V^{\varepsilon}_t(Z^{\epsilon}_u(\omega,y,t)))}du.
\end{align*}
Hence, by \eqref{eq:unif_bound_vepislon},
\begin{align}\label{eq:unif_bound_t}
    \sup_{(0,T)\times \Omega\times \IR^d}\abs{\partial_t\Phi^{\varepsilon}_t(\omega,y)}\leq \sup_{(0,T)\times \Omega\times \IR^d}\abs{\Psi_t(\omega,y)}+T
\sup_{(0,T)\times \Omega\times \IR^d}\norm{\nabla\Psi_t(\omega,y)}\sup_{(0,T)\times \Omega\times \IR^d}\norm{V^{\varepsilon}_t(\omega,y)}<\infty.
\end{align}
For $r>0$ let $\eta_{r}\in C^{\infty}_c((0,T))$ such that \begin{align}\label{eq:int_by_parts}
    0\leq \eta_r\leq 1, \; \lim_{r\to 0}\eta_r(t)=\eins_{(0,T)}(t)  
\end{align}
and such that for all continuous $g:[0,T]\to \IR$ we have 
\[
\int_0^T(\partial_t\eta_r(t))g(t)dt\xrightarrow{r\to 0}g(0)-g(T).
\]
Furthermore, let\begin{align*}
    \chi_R=\chi(\cdot/R)\in C^{\infty}_c(\IR^d)
\end{align*}
be  a smooth  function
such that $0\leq \chi_R\leq 1$, $\norm{\nabla \chi_R}\leq 2/R$, $\chi_R=1$ on $B_R(0)$ and $\chi_R=0$ on $\IR^d\setminus B_{2R}(0)$.
Then it follows from \eqref{eq:unif_bound_grad} and \eqref{eq:unif_bound_t} that $\eta_r(t)\chi_R(y)\Phi_t^{\varepsilon}(\omega,y)\in \test$.

\medskip

Step 4: Using the continuity equation. Set $\pi_2(x,y)=y$ for $(x,y)\in \IR^d\times \IR^d$. Since $((\IP_t)_{t\in [0,T]},(V_t)_{t\in [0,T]})$ satisfies \eqref{eq:CE}\begin{align*}
    0&=\int_0^T\int \partial_t\left(\eta_r(t)\chi_R(y)\Phi^{\varepsilon}_t(\omega,y)\right)d\IP_t(\omega,y)dt\\
    &+\int_0^T\int \nabla\left(\eta_r(t)\chi_R(y)\Phi^{\varepsilon}_t(\omega,y)\right)V_t(\omega,y)d\IP_t(\omega,y)dt\\
    &=\int_0^T\int \eta_r(t)\chi_R(y)\left(\partial_t\Phi^{\varepsilon}_t(\omega,y)+\nabla\Phi^{\varepsilon}_t(\omega,y)V_t(\omega,y)\right)d\IP_t(\omega,y)dt\\
    &+\int_0^T\int \eta_r(t)\nabla_{\IR^d}\chi_R(y)\Phi^{\varepsilon}_t(\omega,y)\pi_2\left(V_t(\omega,y)\right)d\IP_t(\omega,y)dt\\
    &+\int_0^T\int \partial_t\eta_r(t)\Phi^{\varepsilon}_t(\omega,y)\chi_R(y)d\IP_t(\omega,y)dt.
\end{align*}
Note that the map 
 $t\mapsto \int \Phi^{\varepsilon}_t(\omega,y)\chi_R(y)d\IP_t(\omega,y)$ is continuous, since\begin{align*}
     &\int \Phi^{\varepsilon}_{t+h}(\omega,y)\chi_R(y)d\IP_{t+h}(\omega,y)\\&=
     \int(\Phi^{\varepsilon}_{t+h}(\omega,y)-\Phi^{\varepsilon}_{t}(\omega,y))\chi_R(y)d\IP_{t+h}(\omega,y)+\int\Phi^{\varepsilon}_{t}(\omega,y)\chi_R(y)d\IP_{t+h}(\omega,y).
 \end{align*} 
 From \eqref{eq:unif_bound_t} it follows that  there exists a constant $K>0$ such that $\norm{\Phi^{\varepsilon}_{t+h}(\omega,y)-\Phi^{\varepsilon}_{t}(\omega,y)}\leq h K$ and thus \[
 \int(\Phi^{\varepsilon}_{t+h}(\omega,y)-\Phi^{\varepsilon}_{t}(\omega,y))\chi_R(y)d\IP_{t+h}(\omega,y)\xrightarrow{h\to 0}0.
\] By continuity of $\Phi^{\varepsilon}$ and the curve $(\IP_t)_{0\leq t\leq T}$\[
 \int\Phi^{\varepsilon}_{t}(\omega,y)\chi_R(y)d\IP_{t+h}(\omega,y)\xrightarrow{h\to 0}
 \int\Phi^{\varepsilon}_{t}(\omega,y)\chi_R(y)d\IP_{t}(\omega,y).
 \]
Then, by the choice of the functions $\eta_r$ and the continuity of the map $t\mapsto \int \Phi^{\varepsilon}_t(\omega,y)\chi_R(y)d\IP_t(\omega,y)$
\begin{align*}
    \int_0^T\int \partial_t\eta_r(t)\Phi^{\varepsilon}_t(\omega,y)\chi_R(y)d\IP_t(\omega,y)dt\xrightarrow{r\to 0}
    &\int \Phi^{\varepsilon}_0(\omega,y)\chi_R(y)d\IP_0(\omega,y)
    -\int \Phi^{\varepsilon}_T(\omega,y)\chi_R(y)d\IP_T(\omega,y)\\
    &= \int \Phi^{\varepsilon}_0(\omega,y)\chi_R(y)d\IP_0(\omega,y),
\end{align*}
where the equality follows from the fact that  $\Phi^{\varepsilon}_T=0$. Hence, letting $r\to 0$ yields 
\begin{align*}
    0&=\int_0^T\int \chi_R(y)\left(\partial_t\Phi^{\varepsilon}_t(\omega,y)+\nabla\Phi^{\varepsilon}_t(\omega,y)V_t(\omega,y)\right)d\IP_t(\omega,y)dt\\
    &+\int_0^T\int \nabla_{\IR^d}\chi_R(y)\Phi^{\varepsilon}_t(\omega,y) \pi_2\left(V_t(\omega,y)\right)d\IP_t(\omega,y)dt\\
    &+\int \Phi^{\varepsilon}_0(\omega,y)\chi_R(y)d\IP_0(\omega,y).
\end{align*}

Since $\Phi^{\varepsilon}$ solves the equation \eqref{eq:transport_psi} \begin{align*}
   0&=\int_0^T\int \chi_R(y)\left(\Psi_t(\omega,y)+\nabla\Phi^{\varepsilon}_t(\omega,y)[V_t(\omega,y)-V^{\epsilon}_t(\omega,y)]\right)d\IP_t(\omega,y)dt\\
    &+\int_0^T\int \nabla_{\IR^d}\chi_R(y)\Phi^{\varepsilon}_t(\omega,y) \pi_2\left(V_t(\omega,y)\right)d\IP_t(\omega,y)dt\\
    &+\int \Phi^{\varepsilon}_0(\omega,y)\chi_R(y)d\IP_0(\omega,y).
\end{align*}
Since $\Phi_0\leq 0$ and $\IP_0\leq 0$ this implies \begin{align*}
    0&\geq\int_0^T\int \chi_R(y)\left(\Psi_t(\omega,y)+\nabla\Phi^{\varepsilon}_t(\omega,y)[V_t(\omega,y)-V^{\epsilon}_t(\omega,y)]\right)d\IP_t(\omega,y)dt\\
    &+\int_0^T\int \nabla_{\IR^d}\chi_R(y)\Phi^{\varepsilon}_t(\omega,y) \pi_2\left(V_t(\omega,y)d\IP_t(\omega,y)\right)dt\\
    &\geq \int_0^T\int \chi_R(y)\left(\Psi_t(\omega,y)+\nabla\Phi^{\varepsilon}_t(\omega,y)[V_t(\omega,y)-V^{\epsilon}_t(\omega,y)]\right)d\IP_t(\omega,y)dt\\
    &-T\int_0^T\int \norm{\nabla_{\IR^d}\chi_R(y)}
\norm{V_t(\omega,y)}d\abs{\IP}_t(\omega,y)dt,
\end{align*}
where in the last line we used that $0\geq \Phi^{\epsilon}\geq -T$.

From the uniform  bound \eqref{eq:unif_bound_grad} on $\nabla\Phi^{\varepsilon}_t(\omega,y)$ and the fact that $V_t(\omega,y)=V_t'(\omega,y)$ on $[0,T]\times\supp{\chi_R}\subset [0,T]\times B_{2R}(0)$, see \eqref{eq:def_V'}, we obtain for $\varepsilon\to 0$\begin{align*}
    &\int_0^T\int \chi_R(y)\nabla\Phi^{\varepsilon}_t(\omega,y)[V_t(\omega,y)-V^{\epsilon}_t(\omega,y)]d\IP_t(\omega,y)dt\to 0.
\end{align*}
Hence \begin{align*}
    \int_0^T\int \chi_R(y)\Psi_t(\omega,y)d\IP_t(\omega,y)dt&\leq T\int_0^T\int \norm{\nabla_{\IR^d}\chi_R(y)}\norm{V_t(\omega,y)}d\abs{\IP_t}(\omega,y)dt\\
    &\leq \frac{2T}{R}\int_0^T\int \norm{V_t(\omega,y)}\eins_{1\leq \norm{y}\leq 2R}d\abs{\IP_t}(\omega,y)dt\\
    &\leq \frac{2T}{R}\int_0^T\int \norm{V_t}d\abs{\IP_t}dt.
\end{align*}
Letting $R\to \infty$ yields 
\[
\int_0^T\int \Psi_t(\omega,y)d\IP_t(\omega,y)dt\leq 0.
\]
By arbitrariness of $\Psi(\geq 0)$ this yields $\IP_t\leq 0$ for all $0\leq t\leq T$.
%Applying the same reasoning to the curve $(-\IP_t)_{0\leq t\leq T}$, which also solves the continuity equation \eqref{eq:CE}, yields $\IP_t=0$ for all $0\leq t\leq T$. 
\end{proof}

\begin{prop}\label{prop:repres_form}
Let $((\IP_t)_{t\in [0,T]},(V_t)_{t\in [0,T]})$ satisfy \eqref{eq:CE} and $(\IP_t)_{t\in [0,T]}$ be a curve of probability measures.   Assume that for all compact sets $B\subset \IR^d$ \begin{align}\label{eq:usual_assumptions}
        \int_0^T\int \norm{V_t}d\IP_tdt<\infty
    \text{ and }
        \int_0^T\sup_{\Omega\times B}\norm{V_t}+Lip(V_t,\Omega\times B)dt<\infty.
    \end{align}
    For $\omega\in \Omega$ and $t\in [0,T]$ define $G_t^{\omega}:\IR^d\times \IR^d\to \IR^d\times \IR^d $ by $G^{\omega}_t(x,y)=V_t(\theta_x\omega,y)$. 
    Then for $\IP_0$-a.e. $(\omega,y)\in \Omega\times \IR^d$ the solution of the ODE \begin{align}\label{eq:ODE}
        \partial_t X_t^{\omega}(0,y)=G^{\omega}_t(X_t^{\omega}(0,y)) \text{ and } X_0^{\omega}(0,y)=(0,y)
    \end{align} 
    is defined on the time interval $[0,T]$. Furthermore, \begin{align}
        (Z_t)_{\#}\IP_0=\IP_t,
    \end{align}
    where  $Z_t(\omega,y)=(\theta_{U_t^{\omega}(0,y)}\omega,W_t^{\omega}(0,y))$ and $U^{\omega}$ and $W^{\omega}$ are defined by $X_t^{\omega}(0,y)=(U_t^{\omega}(0,y),W_t^{\omega}(0,y))$.
\end{prop}
\begin{proof}
    Let $\tau(\omega,y)$ be the length of the time interval of the maximal solution of the ODE \eqref{eq:ODE} and set $E_u=\{(\omega,y)\in \Omega\times \IR^d: \tau(\omega,y)>u\}$. We claim that the curve $[0,u]\ni t\mapsto (Z_t)_{\#}(\eins_{E_u}\IP_0)$ solves the continuity equation \eqref{eq:CE} w.r.t. the vector fields $(V_t)_{0\leq t\leq u}$. Let $\Phi\in \test$, then for $(\omega,y)\in E_u$ and $t<u$ \begin{align*}
        \partial_t \Phi_t(Z_t(\omega,y))&=\partial_t \Phi_t(\theta_{U_t^{\omega}(0,y)}\omega,W_t^{\omega}(0,y))\\
        &=[\partial_t \Phi_t](\theta_{U_t^{\omega}(0,y)}\omega,W_t^{\omega}(0,y))+\inner{[\nabla_{\Omega}\Phi_t](\theta_{U_t^{\omega}(0,y)}\omega,W_t^{\omega}(0,y))}{\partial_t U_t^{\omega}(0,y)}\\
        &+\inner{[\nabla_{\IR^d}\Phi_t](\theta_{U_t^{\omega}(0,y)}\omega,W_t^{\omega}(0,y))}{\partial_t W_t^{\omega}(0,y)}\\
        &= [\partial_t \Phi_t](Z_t(\omega,y))+ \inner{[\nabla\Phi_t](Z_t(\omega,y))}{G^{\omega}_t(X_t^{\omega}(0,y))}\\
         &= [\partial_t \Phi_t](Z_t(\omega,y))+ \inner{[\nabla\Phi_t](Z_t(\omega,y))}{V_t(Z_t(\omega,y))}.
    \end{align*}
We obtain \begin{align*}
    &\int_0^1\int_{\Omega\times \IR^d} \partial_t \Phi_t + \inner{V_t} {\nabla \Phi_t} d(Z_t)_{\#}(\eins_{E_u}\IP_0) dt\\
    &=\int_{E_u}\int_0^1 [\partial_t \Phi_t](Z_t(\omega,y))+ \inner{[\nabla\Phi_t](Z_t(\omega,y))}{V_t(Z_t(\omega,y))} dt d\IP_0(\omega,y)\\
     &=\int_{E_u}\int_0^1 \partial_t \Phi_t(Z_t(\omega,y)) dt d\IP_0(\omega,y)\\
     &= \int_{E_u} \Phi_1(Z_1(\omega,y))- \Phi_0(Z_0(\omega,y))  d\IP_0(\omega,y)\\
     &=0,
\end{align*}
where the last equality follows from the fact that $\Phi_1=\Phi_0=0$, since $\Phi\in \test$.

Hence, Proposition \ref{prop:comparison} implies that for all $0\leq t\leq u$ \begin{align}\label{eq:repres_compa}
    (Z_t)_{\#}(\eins_{E_u}\IP_0)\leq \IP_t
\end{align}and we have\begin{align*}
    \int \sup_{0\leq t\leq \tau}\norm{X_t^{\omega}(0,y)-(0,y)}d\IP_0(\omega,y)&\leq \int \int_0^{\tau(\omega,y)}\norm{\partial_tX^{\omega}_t(0,y)}dtd\IP_0(\omega,y)\\
    &= \int \int_0^{\tau(\omega,y)}\norm{G^{\omega}_t(X^{\omega}_t(0,y))}dtd\IP_0(\omega,y)\\
    &=\int \int_0^{\tau(\omega,y)}\norm{V_t(Z_t(\omega,y))}dtd\IP_0(\omega,y)\\
    &=\int_0^T\int_{E_t}\norm{V_t(Z_t(\omega,y))}d\IP_0(\omega,y)dt\\
    &\stackrel{\eqref{eq:repres_compa}}{\leq}
    \int_0^T\int\norm{V_t(\omega,y)}d\IP_t(\omega,y)dt\\
    &< \infty.
\end{align*}
    In particular, for $\IP_0$-a.e. $(\omega,y)$ the maximal solution of the ODE \eqref{eq:ODE} is bounded on its domain and hence it is defined on $[0,T]$. This implies that $E_u=\Omega\times \IR^d$ for $0\leq u<T$ and hence \eqref{eq:repres_compa} shows that $(Z_t)_{\#}\IP_0\leq \IP_t$ for  $0\leq t\leq T$.
    Since $(Z_t)_{\#}\IP_0$ and $\IP_t$ are probability measures, this yields  $(Z_t)_{\#}\IP_0= \IP_t$ for  $0\leq t\leq T$.
\end{proof}

Let $AC((0,T), \IR^d\times \IR^d)$ be the space of absolutely continuous curves $\gamma:(0,T)\to \IR^d\times \IR^d$. The following superposition principle holds.

\begin{cor}\label{cor:superpositionpr}
Let $((\IP_t)_{t\in [0,T]},(V_t)_{t\in [0,T]})$ satisfy \eqref{eq:CE}, where     $(\IP_t)_{t\in [0,T]}$ is a  curve of probability measures and the vector field   $(V_t)_{t\in [0,T]}$ satisfies the assumptions \eqref{eq:usual_assumptions}.  
    Then there exists a probability measure $\Lambda$ on $\Omega \times C([0,T], \IR^d\times\IR^d)$ such that 
    \begin{enumerate}[i)]
    \item The measure $\Lambda$ is concentrated on the set of pairs $(\omega,(u_t,w_t)_{t\in [0,T]})$
    such that $(u_t,w_t)_{t\in [0,T]}\in AC((0,T), \IR^d\times \IR^d)$ is a solution of the ODE $(\dot{u}_t,\dot{w}_t)=G^{\omega}_t(u_t,w_t)$ for $\Leb$-a.e. $t\in (0,T)$, with $y_0=0$;
    \item $\IP_t=(F_t)_{\#}\Lambda$, where $F_t:\Omega\times  C([0,T], \IR^d\times \IR^d)\to \Omega\times \IR^d, (\omega,(u_t,w_t)_{t\in [0,T]})\mapsto (\theta_{u_t}\omega,w_t)$.
\end{enumerate}
\end{cor}
\begin{proof}
    We use the same notation as in Proposition \ref{prop:repres_form}. Define the measure $\Lambda$ as the pushforward of $\IP_0$ under the map \begin{align*}
       \Omega\times \IR^d \ni(\omega,x)\mapsto (\omega,(X^{\omega}_t(0,x))_{0\leq t\leq T})\in \Omega \times C([0,T], \IR^d\times\IR^d).
    \end{align*}
    Proposition \ref{prop:repres_form} then implies the listed properties of $\Lambda$.
\end{proof}

We obtain the following formula for the cost function of two random measures.
\begin{thm}\label{prop:Benamou-Brenier}
   For two equivariant random measures $\xi$ and $\eta$ defined on $\Omega$ with the same finite intensity, we have \begin{align}
        \cost_p(\xi,\eta)=\inf \int_0^1 \norm{V_t}_{L^p(\IP_t)}^pdt,
    \end{align}
    where the infimum runs over all solutions of the  continuity equation $((\IP_t)_{t\in [0,1]},(V_t)_{t\in [0,1]})$, where $(\IP_t)_{t\in [0,1]}$ is a curve of probability measures with $(\proj_{\Omega})_{\#}\IP_0=\Q_{\xi}$ and  $(\proj_{\Omega})_{\#}\IP_1=\Q_{\eta}$ and  where $(V_t)_{t\in [0,1]}$ satisfies  the assumptions \eqref{eq:usual_assumptions}.
\end{thm}
\begin{proof}
    By Proposition \ref{prop:wass_dist_palm} 
    \begin{align}
        \cost_p(\xi,\eta)=\inf_{\IP_0,\IP_1}\inf_{\U\in \Cpl(\IP_0,\IP_1)}\IE_{\U}\left[\met^p\right],
    \end{align}
    where the first infimum runs over all probability measures $\IP_i$, $i=0,1$ on $\Omega\times \IR^d$ such that $(\proj_1)_{\#}\IP_0=\Q_{\xi}$ and $(\proj_1)_{\#}\IP_1=\Q_{\eta}$.
    
 Let $((\IP_t)_{0\leq t\leq 1},(V_t)_{0\leq t\leq 1})$ be a solution to the continuity equation \eqref{eq:CE} with $(\proj_1)_{\#}\IP_0=\Q_{\xi}$ and $(\proj_1)_{\#}\IP_1=\Q_{\eta}$, satisfying the assumptions in the statement of this proposition.
 Applying Corollary \ref{cor:superpositionpr} to the curve $(\IP_t)_{0\leq t\leq 1}$ yields
  \begin{align*}
 \mathbb W_p^p(\IP_{0},\IP_{1})&\leq 
 \IE_{\Lambda}\left[\norm{(u_1,w_1)-(u_0,w_0)}^p\right]\\
 &\leq \int_0^1\IE_{\Lambda}\left[\norm{G^{\omega}_t(u_t,w_t)}^p \right]dt\\
 &=\int_0^1\norm{V_t}_{L^p(\IP_t)}^p dt,
 \end{align*}
which proves $\cost_p(\xi,\eta)\leq \inf \int_0^1 \norm{V_t}_{L^p(\IP_t)}^p$.
Assume that $\cost_p(\xi,\eta)<\infty$ and let $T$ be a locally finite kernel satisfying \eqref{eq:transp_prop_kernel}. 
Define the measures $\IP_t$ for $0\leq t\leq 1$ by \[
\IE_{\IP_t}[f]=\IE_{\IP_{\xi}}\left[\int f(\theta_{tz}\omega,z)T(\omega,dz)\right],
\]
where $f:\Omega\times \IR^d\to \IR$ is bounded and measurable. The curve $(\IP_t)_{0\leq t\leq 1}$ is weakly continuous and  $(\proj_1)_{\#}\IP_1=\Q_{\eta}$. Define $V_t(\omega,x)=(x,0)\in \IR^{2d}$. Then $((\IP_t)_{t\in [0,1]},(V_t)_{t\in [0,1]})$ satisfies \eqref{eq:CE},  since for $\Phi\in \test$ \begin{align}\label{eq:CE_geodes}
    &\int_0^1\int_{\Omega\times \IR^d} \partial_t \Phi_t + \inner{V_t} {\nabla \Phi_t} d\IP_t dt\\
    &= \int_0^1\IE_{\IP_{\xi}}\left[\int \left[\partial_t \Phi_t + \inner{V_t} {\nabla \Phi_t}\right](\theta_{tz}\omega,z)T(\omega,dz)\right]dt\nonumber\\
 &= \IE_{\IP_{\xi}}\left[\int \int_0^1\left[\partial_t \Phi_t + \inner{V_t} {\nabla \Phi_t}\right](\theta_{tz}\omega,z) dt T(\omega,dz)\right].\nonumber
\end{align}
Since \[
\frac{d}{dt}\Phi_t (\theta_{tz}\omega,z)=\left[\partial_t \Phi_t + \inner{V_t} {\nabla \Phi_t}\right](\theta_{tz}\omega,z),
\]
and $\Phi_0=\Phi_1=0$,
\eqref{eq:CE_geodes} is equal to $0$.  Furthermore, 
   \begin{align*}
    \norm{V_t}_{L^p(\IP_t)}^p=\IE_{\IP_{\xi}}\left[\int \norm{V_t(\theta_{tz}\omega,z)}^pT(\omega,dz)\right]=\IE_{\IP_{\xi}}\left[\int \norm{z}^pT(\omega,dz)\right]
\end{align*}
and thus \[
 \int_0^1\norm{V_t}_{L^p(\IP_t)}^p dt=\IE_{\IP_{\xi}}\left[\int \norm{z}^pT(\omega,dz)\right].
\]Taking the infimum over locally finite kernels $T$ satisfying \eqref{eq:transp_prop_kernel}, it follows from Corollary \ref{cor:palm_transp} that  \[
c_p(\xi,\eta)\geq \inf \int_0^1\norm{V_t}_{L^p(\IP_t)}^p dt,
\]
which yields the claim.
\end{proof}

\begin{prop}\label{prop:repres_xi_t}
    Let $((\IP_t)_{t\in [0,1]},(V_t)_{t\in [0,1]})$ satisfy \eqref{eq:CE}, where $(\IP_t)_{t\in [0,1]}$ is a curve of probability measures  and  where $(V_t)_{t\in [0,1]}$ satisfies  the assumptions \eqref{eq:usual_assumptions}.
     Assume that $(\proj_{\Omega})_{\#}\IP_0=\Q_{\xi}$ and $(\proj_{\Omega})_{\#}\IP_1=\Q_{\eta}$.  Then   there exist equivariant random measures $\xi_t$ on $\Omega$ such that $(\proj_{\Omega})_{\#}\IP_t=\Q_{\xi_t}$. Furthermore, we have  $\eta=\xi_1$ and the curve $((\xi_t)_{\#}\Q)_{0\leq t\leq 1}$ is weakly continuous on $\mathcal M(\IR^d)$.
\end{prop}
\begin{proof}
Let $\Lambda$ be the probability measure on $\Omega \times C([0,1], \IR^d\times\IR^d)$, constructed in Corollary \eqref{cor:superpositionpr}. Disintegration with respect to the marginal on $\Omega$ then yields a kernel $T$ such that for all measurable  functions $f:\Omega\times \IR^d\to [0,\infty)$\begin{align*}
    \IE_{\IP_t}[f]=\IE_{\IP_0}\left[\int f(\theta_{u_t}\omega,w_t)T(\omega,d (u_s,w_s)_{0\leq s\leq 1})\right]=\IE_{\IP_0}\left[\int f(\theta_u\omega,w)T_t(\omega,d(u,w))\right],
\end{align*}
where the kernel $T_t$ is defined for all $\omega\in \Omega$ by $\int g(u,w) T_t(\omega,d(u,w))=\int g(u_t,w_t)T(\omega,d(u_s,w_s)_{0\leq s\leq 1})$, with $g:\IR^d\times \IR^d\to [0,\infty)$  measurable.

Let $\IP'_t=(\proj_{\Omega})_{\#}\IP_t$.
Then 
\begin{align*}
    \IE_{\IP'_t}[f]=\IE_{\IP_0}\left[\int f(\theta_u\omega)T_t(\omega,d(u,w))\right]=\IE_{\IP_0'}\left[\int f(\theta_u\omega)T_t'(\omega,du)\right],
\end{align*}
where for fixed $\omega\in \Omega$ the kernel $T_t'(\omega,du)$ is defined as the projection on the first marginal of $T_t(\omega,d(u,w))$, that is \begin{align*}
    T_t'(\omega,A)=T_t(\omega,A\times \IR^d),\quad \forall \omega\in \Omega, A\in \mathcal B(\IR^d).
\end{align*}
By Remark \ref{rem:construction kernel} there exists an equivariant random measure $\xi_t$ such that $T_t'$ is $(\xi,\xi_t)$-balancing and such that $\IP_t'=\Q_{\xi_t}$.
We first prove weak continuity of the distributions $((\xi_t)_{\#}\Q)_{0\leq t\leq 1}$. Let $r< t$ and note that the curve $(\IP_s)_{r\leq s\leq t}$ also solves the continuity equation w.r.t. the vector fields $(V_s)_{r\leq s\leq t}$. Let $\tilde{\Lambda}$ be the probability measure obtained by applying the superposition principle Corollary \ref{cor:superpositionpr} to the restricted curve $(\IP_s)_{r\leq s\leq t}$.
Corollary \ref{cor:palm_transp} thus yields the bound
\begin{align*}
    \cost_p(\xi_r,\xi_t)\leq \IE_{\tilde{\Lambda}}\left[\norm{(u_{t-r},w_{t-r})-(u_0,w_0)}^p\right] &\leq \int_0^{t-r}\IE_{\tilde{\Lambda}}\left[\norm{(\dot{u}_s,\dot{w}_s)}^p\right]ds\\
    &=\int_r^t  \norm{V_s}_{L^p(\IP_s)}^p        ds.
\end{align*}
Hence $\cost_p(\xi_r,\xi_t)\to 0$ as $r\to t$. By \cite[Lemma 5.4]{erbar2023optimal} this implies weak continuity of the distributions $((\xi_t)_{\#}\Q)_{0\leq t\leq 1}$. 

We now prove that  $\xi_1=\eta$. Let $f:\Omega\times \IR^d\to [0,\infty)$ be  measurable. Then by the Campbell formula \begin{align}\label{eq:campbell_applied}
    \IE_{\Q}\left[\int f(\omega,x)\xi_1(dx)\right]&=
    \IE_{\Q_{\xi_1}}\left[\int f(\theta_{-x}\omega,x)dx\right]\\
    &=\IE_{\IP_1'}\left[\int f(\theta_{-x}\omega,x)dx\right]\nonumber\\
    &=\IE_{\Q}\left[\int f(\omega,x)\eta(dx)\right].\nonumber
\end{align} 

 For a measurable and bounded set $B\subset \IR^d$ define the function \[
f(\omega,x)=\eins_{\xi_1(B)>\eta(B)}(\omega)\eins_B(x),
\]
plugging the functions $f$ inside \eqref{eq:campbell_applied}, we obtain \[
\IE_{\Q}\left[\eins_{\xi_1(B)>\eta(B)}\xi_1(B)\right]=\IE_{\Q}\left[\eins_{\xi_1(B)>\eta(B)}\eta(B)\right]
.\] 
Hence, $\xi_1=\eta$ $\Q$-a.s..
\end{proof}

% We use the notation of Example \ref{ex:canonical_space}. Recall the definition of the distance $\W_p(\P_0,\P_1)$ for two (distributions of) random measures $\P_0$ and $\P_1$ from \cite{erbar2023optimal}:
% \begin{align}\label{eq:2ndrep}
%      \mathsf W^p_p (\P_0,\P_1)=\inf_{(\xi,\eta)}\cost_p(\xi,\eta),
%  \end{align} 
%  where the infimum runs over all equivariant random measures $\xi$ and $\eta$, defined on an admissible probability space, such that $\xi\sim\P_0, \eta\sim\P_1$. 
  
  Finally, the following Corollary is a precise formulation of our main result, Theorem \ref{thm:main}.
  
\begin{cor}\label{cor:BenBrenier_lawsPP}
  For two stationary random measures $\P_i$, $i=0,1$, with the same finite intensity, we have \begin{align}
        \mathsf W^p_p(\PP_0,\PP_1)=\inf \int_0^1 \norm{V_t}_{L^p(\IP_t)}^pdt,
    \end{align}
    where the infimum runs over all admissible probability spaces $(\Omega,\mathcal F,\Q)$ and 
    curves of equivariant random measures $(\xi_t)_{0\leq t\leq 1}$, defined on $\Omega$, with the following properties 
    \begin{enumerate}
    \item 
      $(\proj_{\Omega})_{\#}\IP_t=\Q_{\xi_t}$, $0\leq t\leq 1$, where
          $((\IP_t)_{t\in [0,1]},(V_t)_{t\in [0,1]})$ satisfies \eqref{eq:CE},   such that $(\IP_t)_{t\in [0,1]}$ is a curve of probability measures  and  where $(V_t)_{t\in [0,1]}$ satisfies  the assumptions \eqref{eq:usual_assumptions}.
        \item The curve of distributions $((\xi_t)_{\#}\Q)_{t\in [0,1]}$ is weakly continuous.
        \item $\xi_0\underset{\Q}{\sim}\P_0$ and $\xi_1\underset{\Q}{\sim}\P_1$      
    \end{enumerate}
\end{cor}
\begin{proof}
    This follows from Propositions \ref{prop:Benamou-Brenier} and \ref{prop:repres_xi_t}.
\end{proof}

  \iffalse
  Define for $\bar{\xi},\bar{\eta}\in \mathcal M(\IR^d)$ \begin{align}\label{eq:intro-ground-cost}
c_p(\bar{\xi},\bar{\eta})=\inf_{\mathsf q\in\cpl(\bar{\xi},\bar{\eta})}  \limsup_{n\to\infty}\frac{1}{n^d} \int_{\Lambda_n\times \IR^d}\norm{x-y}^p \mathsf q(dx,dy),
\end{align}
where $\cpl(\bar{\xi},\bar{\eta})$ is the set of all couplings of 
$\bar{\xi}$ and $\bar{\eta}$.
\begin{defi}
    For two stationary random measures $\P_i$, $i=0,1$, with the same finite intensity, set\begin{align*}
        \C_p(\P_0,\P_1)=\inf_{\mathsf Q\in \Cpl_s(\P_0,\P_1)} \int c_p(\xi,\eta) \mathsf Q(d\xi,d\eta).
    \end{align*}
\end{defi}
\MH{I think we do not need this definition since the metric $\mathsf W_p$ has been defined in the introduction. Wafür brauchen wir \eqref{eq:intro-ground-cost}? Und die nächste Proposition?}
\BM{Hier wird $\C$ nochmal über die Kostenfunktion $c_p$ definiert, also es wird nicht die equivariante Definition gegeben, sondern über den limsup.
Wenn man das nicht erwähnen will, kann man hier die Definition und die Proposition weglassen}
Then \cite[Proposition 1.2]{erbar2023optimal} yields the following equivalent formulation in terms of the cost introduced in Definition \ref{def:cost_fct}.
\begin{prop}
\begin{align}\label{eq:2ndrep}
     \mathsf C_p (\P_0,\P_1)=\inf_{(\xi,\eta)}\cost_p(\xi,\eta),
 \end{align} 
 where the infimum runs over all equivariant random measures $\xi$ and $\eta$, defined on an admissible probability space, such that $\xi\sim\P_0, \eta\sim\P_1$. 
\end{prop}
\fi

%\bibliographystyle{alpha} 
%\printbibliography
%\bibliography{bib.bib}

\end{document}